\documentclass[11pt]{article}

\usepackage{amsmath,amssymb,amsthm,mathtools,setspace,graphicx,array,tikz,url}
\usepackage[text={6.5in,8.4in}, top=1.3in]{geometry}
\newtheorem{theorem}{Theorem}[section]
\newtheorem{proposition}[theorem]{Proposition}
\newtheorem{lemma}[theorem]{Lemma}
\newtheorem{corollary}[theorem]{Corollary}
\newtheorem{conjecture}[theorem]{Conjecture}

\newtheorem{definition}[theorem]{Definition}

\newcommand{\F}{{\mathcal F}}
\newcommand{\eps}{{\varepsilon}}
\newcommand{\R}{{\mathbb R}}

\begin{document}
	
\title{
	Almost $k$-union closed set systems
}
	
\author{
	Raphael Yuster
	\thanks{Department of Mathematics, University of Haifa, Haifa 3498838, Israel. Email: raphael.yuster@gmail.com\,.}
}
	
\date{}
	
\maketitle
	
\setcounter{page}{1}
	
\begin{abstract}
	In a recent breakthrough, Gilmer proved the union closed conjecture up to a constant factor.
	Using Gilmer's method and additional ideas, Chase and Lovett proved an optimal result for almost union-closed set systems. Here that result is extended to higher order unions.
%

\end{abstract}

\section{Introduction}

In a recent breakthrough, Gilmer \cite{gilmer-2022} established the well-known Frankl's union closed conjecture up to a constant factor. Shortly after, that factor has been improved by several authors, pushing Gilmer's method to $\frac{3-\sqrt{5}}{2} \approx 0.3819$ \cite{AHS-2022,CL-2022,pebody-2022,sawin-2022}.
A variation of Gilmer's method improved the constant slightly to $\approx 0.3824$
\cite{cambie-2022,sawin-2022,yu-2022}. Additional ideas may be needed to push the bound further \cite{cambie-2022,ellis-2022,sawin-2022}.
Interestingly, $\frac{3-\sqrt{5}}{2}$ has been shown by Chase and Lovett \cite{CL-2022} to be the {\em optimal}
constant for the {\em approximate} version of the union closed conjecture. Here we show that
the method of Gilmer, and the result of Chase and Lovett, can be extended to the approximate version for higher order unions.

\begin{definition}[Approximate $k$-union closed set system]\label{def:1} Let $k \ge 2$ be an integer and let
$0 \le c \le 1$. A finite set system $\F$ is $c$-approximate $k$-union closed if for at least a $c$-fraction of the $k$-tuples  $A_1,\ldots,A_k \in \F$ we have $\cup_{i=1}^k A_i \in \F$.
\end{definition}
Following Chase and Lovett, we say (informally) that $\F$ is almost $k$-union closed (or almost union-closed when $k=2$) if it is $1-o(1)$ approximate $k$-union closed. Observe also that if $c=1$ in Definition \ref{def:1}, then $\F$ is already union closed, but being almost $k$-union closed is a considerably weaker requirement than being union closed,
and different from being almost $k'$-union closed for $k' \neq k$.

The union closed conjecture asserts that in any finite union-closed set system $\F \neq \{\emptyset\}$ (i.e., nontrivial set systems corresponding to $k=2$ and $c=1$ in Definition \ref{def:1}), there exists an element that belongs to at least half of the sets in $\F$.
Gilmer \cite{gilmer-2022} proved this holds for $\F \subseteq 2^{[n]}$ with the constant $0.01$ instead of $\frac{1}{2}$. 
Let $\psi = \frac{3-\sqrt{5}}{2} \approx 0.3819$; Gilmer conjectured that the method in \cite{gilmer-2022} can yield the constant $\psi$.
As mentioned earlier, this was proved shortly after by several researchers.
In fact, Chase and Lovett \cite{CL-2022} proved that $\psi$ is the optimal constant for almost union-closed set systems.
\begin{theorem}[\cite{CL-2022}]\label{t:cl}
Let $\F \subseteq 2^{[n]}$, $\F \neq \{\emptyset\}$ be a $(1-\eps)$-approximate union closed set system, where $0 \le \eps <\frac{1}{2}$. Then there is an element contained in a $\psi-\delta$ fraction of sets in $\F$, where $\delta = 2\eps \left(1+\log \frac{1}{\epsilon}/{\log |\F|}\right)$.
Moreover, for every $n$, there exists such an $\F$ which is $1-o_n(1)$ approximate union closed
and in which every element is contained in at most $\psi+o_n(1)$ sets.
\end{theorem}
\noindent
As we shall see, Theorem \ref{t:clg} below implies the following theorem:
\begin{theorem}\label{t:main}
	Let $\F \subseteq 2^{[n]}$, $\F \neq \{\emptyset\}$ be a $(1-\eps)$-approximate $k$-union closed set system, where $0 \le \eps <\frac{1}{2}$. Then there is an element contained in a $\frac{\ln k}{3k}-\delta$ fraction of sets in $\F$, where $\delta = \left(k\eps + 2\eps \log \frac{1}{\epsilon}/\log |\F|\right)^{1/(k-1)}$.
\end{theorem}
Let $\psi_k$ be the unique real root of $(1-x)^k-x$ in $[0,1]$ (so $\psi=\psi_2$).
The construction used to prove the second part of Theorem \ref{t:cl}
generalizes to almost $k$-union closed set systems.
\begin{proposition}\label{p:1}
	Let $k \ge 2$ be an integer. For every $n$, there exists $\F \subseteq 2^{[n]}$, $\F \neq \{\emptyset\}$
	such that $\F$ is $1-o_n(1)$ approximate $k$-union closed, while every element is contained in at most $\psi_k+o_n(1)$ sets.
\end{proposition}

The following conjecture asserts that the first part of Theorem \ref{t:cl} also generalizes to almost $k$-union closed set systems.
\begin{conjecture}\label{conj:clg}
	Let $\F \subseteq 2^{[n]}$, $\F \neq \{\emptyset\}$ be a $(1-\eps)$-approximate $k$-union closed set system, where $0 \le \eps <\frac{1}{2}$. Then there is an element contained in a $\psi_k-\delta$ fraction of sets in $\F$, where $\delta = \left(k\eps + 2\eps \log \frac{1}{\epsilon}/\log |\F|\right)^{1/(k-1)}$.
\end{conjecture}
\noindent
Theorem \ref{t:cl} is the case $k=2$ of Conjecture \ref{conj:clg}. We prove the next few cases of Conjecture \ref{conj:clg}.
\begin{theorem}\label{t:clg-345}
	Conjecture \ref{conj:clg} holds for $k=3,4$.
\end{theorem}
We next prove a variant of Conjecture \ref{conj:clg} for {\em all} $k$ where instead of $\psi_k$, we use a smaller constant. Moreover, that constant is close to $\psi_k$ in the sense made precise in the following theorem
(see Table \ref{table:k-vars} for a comparison of $z_k$ and $\psi_k$ for small $k$).
\begin{theorem}\label{t:clg}
	Conjecture \ref{conj:clg} holds with the constant $z_k$ instead of $\psi_k$ where
	$$
	z_k > \frac{\ln k}{3k}\;, \qquad \frac{1}{2} < \frac{z_k}{\psi_k} \le 1\;, \qquad \lim_{k \rightarrow \infty} \frac{z_k}{\psi_k} = \frac{\log \frac{1}{\varphi}}{\log 2} \approx 0.6943\;.
	$$
\end{theorem}

An important ingredient in the proof of Theorem \ref{t:cl} is a generalization of an inequality
stated by Boppana \cite{boppana-1985} concerning the minimum of some function in $[0,1]$ related to binary entropy.
This inequality was proved by Boppana \cite{boppana-2023} and by Alweiss, Huang, and Sellke \cite{AHS-2022}.
Though technical, this generalization can be proved rigorously for $k=3,4$, while for larger $k$,
it can be shown to reduce Conjecture \ref{conj:clg} to a conjecture about roots of certain real polynomials.
Assuming this generalization, the arguments of Gilmer and of Chase and Lovett can be rather smoothly
generalized to yield Theorems \ref{t:clg-345} and \ref{t:clg}.

We proceed to prove Proposition \ref{p:1} in Section \ref{sec:proposition}.
Section \ref{sec:inequality} considers the generalization of the aforementioned inequality of Boppana,
proving certain properties related to it. These properties are then used in Section \ref{sec:multi}
to prove a multidimensional version of the Chase-Lovett main lemma and consequently in Section \ref{sec:proofs} to prove Theorems \ref{t:clg-345} and \ref{t:clg}.

\section{The generalized construction}\label{sec:proposition}

\begin{proof}[Proof of Proposition \ref{p:1}]
The construction is a generalization of the one used by Chase and Lovett \cite{CL-2022}.
Define the following set systems over $[n]$:
$$
\F_1 = \{x \in \{0,1\}^n\,:\, |x| = \lfloor \psi_k n + n^{2/3} \rfloor\}, \quad
\F_2 = \{x \in \{0,1\}^n\,:\, |x| \ge  \lfloor (1-\psi_k) n \rfloor\}
$$
and let $\F = \F_1 \cup \F_2$.
As $\psi_k < \frac{1}{2}$, we obtain that $|\F_2| = o_n(|\F_1|)$.
Clearly, each element is in a $\psi_k+o_n(1)$ fraction of the sets $\F_1$, hence $\F$.
Finally, with probability $1-o_n(1)$, a randomly chosen $k$-tuple of sets of $\F_1$ almost surely has
more than $n(\sum_{j=1}^k(-1)^{j-1}\binom{k}{j}\psi_k^j)=n(1-\psi_k)$ elements where we have used
$(1-\psi_k)^k=\psi_k$. Consequently, a randomly chosen $k$-tuple of sets of $\F$ is almost surely in $\F$,
so $\F$ is $1-o_n(1)$ approximate $k$-union closed.
\end{proof}

\section{An inequality concerning binary entropy}\label{sec:inequality}

Recall that $\psi_k$ denotes the unique real root of $(1-x)^k-x$ in $[0,1]$.
Let
\begin{equation}\label{e:basic}
\varphi_k \coloneqq1 - \psi_k\,, \qquad \alpha_k \coloneqq {\varphi_k}^{k-1}=\frac{1}{\varphi_k}-1 = \frac{\psi_k}{\varphi_k}\,.
\end{equation}
\begin{table}[h]
	\centering
	\begin{tabular}{c||c|c|c|c}
		$k$ &  $\varphi_k$ & $\psi_k$ & $z_k$ & $\alpha_k$ \\
		\hline
		$2$ & $0.6180$ & $0.3819$ & $0.3819$ & $0.6180$\\
		\hline
		$3$ & $0.6823$ & $0.3176$ & $0.3176$ & $0.4655$\\
		\hline
		$4$ & $0.7244$ & $0.2755$ & $0.2755$ & $0.3802$\\
		\hline
		$5$ & $0.7548$ & $0.2451$ & $0.2416$ & $0.3247$\\
		\hline
		$6$ & $0.7780$ & $0.2219$ & $0.2183$ & $0.2851$\\
		\hline
		$7$ & $0.7965$ & $0.2034$ & $0.2006$ & $0.2554$\\
		\hline
		$8$ & $0.8116$ & $0.1883$ & $0.1863$ & $0.2319$\\
		\hline
		$16$ & $0.8771$ & $0.1228$ & $0.1204$ & $0.1400$
	\end{tabular}
	\caption{The values of $\varphi_k,\psi_k,z_k,\alpha_k$ for several $k$, listed with precision $10^{-4}$.}
	\label{table:k-vars} 
\end{table}
Some values of these parameters are given in Table \ref{table:k-vars}.

Throughout this paper, all logarithms are natural.
Let $h(x)=-x \log x-(1-x)\log(1-x)$ be the binary entropy function defined
continuously in $[0,1]$ by $h(0)=h(1)=0$.
As in \cite{boppana-1985}, it will be convenient to extend $h(x)$ (continuously) to $\R$ as follows:
$$
h(x) \coloneqq
\begin{cases}
	-x \log |x| - (1-x)\log |1-x| & {\rm if}~~ x \in {\mathbb R}\setminus \{0,1\}\,;\\
	0 & {\rm if}~~ x \in \{0,1\}\,.
\end{cases}
$$
For $k \ge 2$, let $r_k(x)$, $s_k(x)$ and $f_k(x)$ be the functions with domain $\R$ defined as:
\begin{equation}\label{e:functions}
r_k(x) \coloneqq h(x^k)\,, \qquad s_k(x) \coloneqq x^{k-1}h(x)\,, \qquad f_k(x) \coloneqq \alpha_k r_k(x) - s_k(x)\,.
\end{equation}
In \cite{AHS-2022,boppana-2023} it is proved that $f_2(x)$ is nonnegative on $[0,1]$.
The proof in \cite{boppana-2023} uses only differential calculus and the proof in \cite{AHS-2022} uses both differential calculus and interval arithmetic.
\begin{conjecture}\label{conj:fk}
	$f_k(x)$ is nonnegative on $[0,1]$.
\end{conjecture}
As we shall see in the following sections, Conjecture \ref{conj:clg} reduces to Conjecture \ref{conj:fk}.
Being non-parameterized, it seems hopeless to extend the interval arithmetic part of the proof in \cite{AHS-2022} to general $k$. On the other hand, as \cite{boppana-2023} uses only differential calculus,
it may not be hopeless to extend its proof to arbitrary $k$.
In fact, we manage to do so completely rigorously for $k=3,4$. The next several lemmata prove properties
of $f_k(x)$, valid for all $k$.

\begin{lemma}\label{l:1}
	$f_k(0)=f_k(1)=f_k(\varphi_k)=f_k'(\varphi_k)=0$.
\end{lemma}
\begin{proof}
	By assignment, $f_k(0)=f_k(1)=0$. We verify the remaining claims:
	\[
	f_k(\varphi_k)=\alpha_k h({\varphi_k}^k)-\alpha_k h(\varphi_k)=
	\alpha_k(h(1-\varphi_k)-h(\varphi_k))=0\,.
	\]
	For $x \in (0,1)$ we have
	\[
	f_k'(x) = \alpha_k kx^{k-1} \log (x^{-k}-1) 	-x^{k-2}((k (x - 1) + 1) \log(1-x) - k x \log x)
	\]
	so we must prove that
	$\alpha_k kx \log (x^{-k}-1) -(k (x - 1) + 1) \log(1-x) + k x \log x$
	vanishes at $x=\varphi_k$. Indeed, substituting $x$ with $\varphi_k$ in the last expression we obtain
	\begin{align*}
	& \; k \alpha_k\varphi_k \log ({\varphi_k}^{-k}-1) -(k (\varphi_k - 1) + 1) \log(1-\varphi_k) + k \varphi_k \log \varphi_k\\
	= & \; k(1-\varphi_k)(\log{\varphi_k}-\log(1-{\varphi_k})) - (k (\varphi_k - 1) + 1) \log(1-\varphi_k) + k \varphi_k\log \varphi_k\\
	= & \; k \log \varphi_k - \log(1-\varphi_k)\\
	= & \;\log(1-\varphi_k) - \log(1-\varphi_k)\\
	= & \; 0\,.
	\end{align*}
\end{proof}

\begin{lemma}\label{l:2}
	$f_k(x)$ is positive in $(0,\eps)$ for some small $\eps > 0$.
\end{lemma}
\begin{proof}
	The Taylor expansion of $\log(1-\eps)$ gives that for all $\eps \in (0,1)$,
	$$
	\eps\left(\log \frac{1}{\eps}+1-\eps\right) \le h(\eps) \le \eps\left(\log \frac{1}{\eps} +1\right)\,.
	$$
	We therefore have
	\begin{align*}
	\alpha_ k r_k(\epsilon) & = \alpha_ k h(\eps^k) \ge \alpha_k \eps^k \left(\log \frac{1}{\eps^k} + 1-\eps^k\right);\\
	s_k(\epsilon) & = \eps^{k-1} h(\eps) \le \eps^k\left(\log \frac{1}{\eps} +1\right).
	\end{align*}
	Dividing both inequalities by $\eps^k$ it remains to prove that for small $\eps > 0$,
	$$
	\alpha_k\left(\log \frac{1}{\eps^k} + 1-\eps^k\right) > \left(\log \frac{1}{\eps} +1\right)\,.
	$$
	Equivalently, we must show that for small $\eps > 0$,
	$$
	\alpha_k > \frac{\log \frac{1}{\eps} +1}{k\log \frac{1}{\eps} + 1 -\eps^k}\;.
	$$
	Since $\alpha_k = \frac{1}{\varphi_k}-1$  it suffices to show that for small $\eps > 0$,
	$$
	\varphi_k < \frac{k\log \frac{1}{\eps} + 1 -\eps^k}{(k+1)\log \frac{1}{\eps} + 2 -\eps^k}\;.
	$$
	We will show the stronger statement that for small $\eps > 0$,
	$$
	\varphi_k < \frac{k\log \frac{1}{\eps}}{(k+1)\log \frac{1}{\eps} + 2}\;.
	$$
	Indeed, notice that since $(1-1/(k+1))^k > 1/(k+1)$, we have that $\varphi_k < k/(k+1)$,
	so for some $0 < \delta < 1$ we have $\varphi_k = \delta k/(k+1)$.
	We may therefore choose  $\eps > 0$ sufficiently small such that
	$$
	\frac{k\log \frac{1}{\eps}}{(k+1)\log \frac{1}{\eps} + 2} > \frac{\delta k}{k+1}=\varphi_k\;.
	$$
\end{proof}

The derivatives of $h(x)$ in $(-1,1) \setminus \{0\}$ are required for the next two lemmas.
By induction, it holds that:
\begin{align}
h'(x) & = \log\left(\frac{1-x}{|x|} \right); \label{e:1}\\
h^{(t)}(x) & = (t-2)!(-1)^t\left( \frac{1}{(x-1)^{t-1}} - \frac{1}{x^{t-1}}\right) ~{\rm for~all}~ t \ge 2\;. \label{e:2}
\end{align}

\begin{lemma}\label{l:3}
	Let $t \ge 0$.\\
	(i) The $t$'th derivative of $s_k(x)$ in $(-1,1) \setminus \{0\}$ is
	\[
	s_k(x)^{(t)} = \sum_{j=0}^{t} h^{(j)}(x) \binom{k-1}{t-j}\frac{t!}{j!}x^{k-t+j-1}\;.
	\]
	(ii) For all $0 \le t \le k-1$, $s_k(0)^{(t)}=0$.\\
	(iii) $s_k(x)^{(k+1)}$ is a rational function in $(0,1)$ given by:
	\begin{align*}
		s_k(x)^{(k+1)}
		& = \sum_{j=0}^{k-1} (-1)^j(k-1)!\binom{k+1}{j+2}\left(\frac{x^{j+1}-(x-1)^{j+1}}{x(x-1)^{j+1}}\right).
	\end{align*}
\end{lemma}
\begin{proof}
	Recall that $s_k(x)=x^{k-1}h(x)$ so (i) is obtained directly by induction and the product rule.
	
	As for (ii), notice first that $s_k(0)=0$.
	Now, suppose $1 \le t \le k-1$, and consider the limit of (i) as $x$ goes to $0$.
	We compute this limit for each term $j$ separately.
	The term corresponding to $j=0$ is just a constant multiple of $h(x)x^{k-t-1}$ so it goes to $0$.
	By \eqref{e:1}, the term corresponding to $j=1$ is a constant multiple of
	$$
	\log\left(\frac{1-x}{|x|} \right)x^{k-t}
	$$
	and since $k -t > 0$, it goes to zero.
	By \eqref{e:2}, the term corresponding to $2 \le j \le t$ is a constant multiple of
	$$
	\left( \frac{1}{(x-1)^{j-1}} - \frac{1}{x^{j-1}}\right)x^{k-t+j-1}= \frac{x^{k-t+j-1}}{(x-1)^{j-1}}-x^{k-t}
	$$
	and since $k -t > 0$, it goes to zero as well.
	
	As for (iii), observe that by (i), the terms involving $h(x)$ and $h'(x)$ vanish, so we are left with
	a rational function, explicitly given by
	\begin{align*}
		s_k(x)^{(k+1)} & = \sum_{j=0}^{k-1} h^{(j+2)}(x) \binom{k-1}{k-1-j}\frac{(k+1)!}{(j+2)!}x^{j}\\
		& = \sum_{j=0}^{k-1} (-1)^j j!\left( \frac{1}{(x-1)^{j+1}} - \frac{1}{x^{j+1}}\right) \binom{k-1}{k-1-j}\frac{(k+1)!}{(j+2)!}x^{j}\\
		& = \sum_{j=0}^{k-1} (-1)^j(k-1)!\binom{k+1}{j+2}\left(\frac{x^{j+1}-(x-1)^{j+1}}{x(x-1)^{j+1}}\right)\;.
	\end{align*}
\end{proof}
	
\begin{lemma}\label{l:4}
	Let $t \ge 0$.\\
	(i) The $t$'th derivative of $r_k(x)$ in $(-1,1) \setminus \{0\}$ is
	\[
    r_k(x)^{(t)} = \sum_{j=0}^{t} (k-1)!C(k,t,j)h^{(j)}(x^k)x^{kj-t}
    \]
    where the coefficient $C(k,t,j)$ satisfies $C(k,0,0)=1/(k-1)!$, otherwise $C(k,t,j)=0$ if $t\cdot j=0$ and otherwise
    \[C(k,t,j)=(kj-t+1)C(k,t-1,j)+kC(k,t-1,j-1)\;.\]
	(ii) For all $0 \le t \le k-1$, $r_k(0)^{(t)}=0$.\\
	(iii) $r_k(x)^{(k+1)}$ is a rational function in $(0,1)$ given by:
	\begin{align*}
	r_k(x)^{(k+1)} & =  \sum_{j=0}^{k-1} (-1)^j j! (k-1)!C(k,k+1,j+2) \left( \frac{x^{kj+k}-(x^k-1)^{j+1}}{x(x^k-1)^{j+1}}\right).
	\end{align*}
\end{lemma}
\begin{proof}
	Recall that $r_k(x)=h(x^k)$ so (i) is obtained directly by induction, the product rule, and the definition
	of the coefficients $C(k,t,j)$. We note that there is no simple ``sum-free'' expression in the general case of $C(k,t,j)$ (e.g., $(k-1)!C(k,6,2)=k^2(k-1)(k-2)(31k^2-132k+137)$), but notice that we do have that for all $1 \le t \le k$,
	$$
	C(k,t,1) = \frac{k}{(k-t)!}
	$$
	and hence $C(k,t,1)=0$ for all $t > k$.
	Also notice that since $C(k,t,j)=0$ when exactly one of $t$ or $j$ is zero, we inductively have that when $0 \le t < j$, 
	$$
	C(k,t,j) = 0\;.
	$$
	As for (ii), notice first that $r_k(0)=0$.
	Now, suppose $1 \le t \le k-1$, and consider the limit of (i) as $x$ goes to $0$.
	We compute this limit for each term $j$ separately.
	The term corresponding to $j=0$ is just $0$.
	By \eqref{e:1}, the term corresponding to $j=1$ is a constant multiple of
	$$
	\log\left(\frac{1-x^k}{|x^k|} \right)x^{k-t}
	$$
	and since $k -t > 0$, it goes to zero.
	By \eqref{e:2}, the term corresponding to $2 \le j \le t$ is a constant multiple of
	$$
	\left( \frac{1}{(x^k-1)^{j-1}} - \frac{1}{x^{kj-k}}\right)x^{kj-t}= \frac{x^{kj-t}}{(x^k-1)^{j-1}}-x^{k-t}
	$$
	and since $k -t > 0$, it goes to zero as well.
	
	As for (iii), observe that by (i), and since $C(k,t,1)=0$ for all $t > k$, we see that in $r_k(x)^{(k+1)}$, the terms involving $h(x)$ and $h'(x)$ vanish, so we are left with a rational function explicitly given by
	\begin{align*}
		r_k(x)^{(k+1)} & = \sum_{j=0}^{k-1} h^{(j+2)}(x^k)(k-1)!C(k,k+1,j+2)x^{kj+k-1}\\
		& = \sum_{j=0}^{k-1} (-1)^j j!\left( \frac{1}{(x^k-1)^{j+1}} - \frac{1}{x^{kj+k}}\right)(k-1)!C(k,k+1,j+2)x^{kj+k-1}\\
		& = \sum_{j=0}^{k-1} (-1)^j j! (k-1)!C(k,k+1,j+2)
		\left( \frac{x^{kj+k}-(x^k-1)^{j+1}}{x(x^k-1)^{j+1}}\right).
	\end{align*}
\end{proof}

\noindent
The following corollary is immediate from Lemma \ref{l:3} item (ii) and Lemma \ref{l:4} item (ii).
\begin{corollary}\label{coro:zero-root}
	$f_k(x)$ has a root of multiplicity $k$ at $x=0$.
\end{corollary}

\noindent 
The following corollary follows from Lemma \ref{l:3} item (iii) and Lemma \ref{l:4} item (iii).
\begin{corollary}\label{coro:pk}
	The $(k+1)$'th derivative of $f_k(x)$ in $(0,1)$ is a rational function of the form
	$(k-1)!p_k(x)/(x(x^k-1)^{k})$ where $p(x)$ is a polynomial of degree $k^2-1$ given by
	$$
	p_k(x) =  \alpha_k\rho_k(x) - \sigma_k(x)
	$$
	where
	\begin{align*}
	\rho_k(x) & = \sum_{j=0}^{k-1} (-1)^j j! C(k,k+1,j+2) \left( (x^k-1)^{k-j-1}x^{kj+k} - (x^k-1)^k \right);\\
	\sigma_k(x) & = \sum_{j=0}^{k-1} (-1)^j \binom{k+1}{j+2}
	\left(x^{j+1}(x-1)^{k-j-1}(1+x+\cdots+x^{k-1})^k - (x^k-1)^k \right)
	\end{align*}
	and where the coefficient $C(k,t,j)$ satisfies $C(k,0,0)=1/(k-1)!$, $C(k,t,0)=0$ if $t > 0$ and otherwise
	$$
	C(k,t,j)=(kj-t+1)C(k,t-1,j)+kC(k,t-1,j-1)\;.
	$$
\end{corollary}
\begin{proof}
	By Lemma \ref{l:3} item (iii) and Lemma \ref{l:4} item (iii) we obtain that
	$$
	f_k(x)^{(k+1)}  = \sum_{j=0}^{k-1} \alpha_k(-1)^j j! (k-1)!C(k,k+1,j+2)
	\left( \frac{x^{kj+k}-(x^k-1)^{j+1}}{x(x^k-1)^{j+1}}\right)-
	$$
	$$
	\sum_{j=0}^{k-1} (-1)^j(k-1)!\binom{k+1}{j+2}\left(\frac{x^{j+1}-(x-1)^{j+1}}{x(x-1)^{j+1}}\right)\;.
	$$
	The common denominator of all terms is $x(x^k-1)^k$, so  $f_k(x)^{(k+1)} = (k-1)!p_k(x)/(x(x^k-1)^k)$ where
	$$
	p_k(x) =  \alpha_k\rho_k(x) - \sigma_k(x)
	$$
	and where $\rho_k(x)$ and $\sigma_k(x)$ are as defined is the statement of the corollary.
	Notice that $\sigma_k(x)$ is of degree $k^2-1$ and $\rho_k(x)$ is of degree $k^2-k$,
	so $p_k(x)$ is of degree $k^2-1$.
\end{proof}

\noindent
{\em Note:} setting $x^k=y$ we can rewrite $\rho_k(x)$ as
$$
\rho_k(x) = \sum_{j=0}^{k-1} (-1)^j j! C(k,k+1,j+2) \left( (y-1)^{k-j-1}y^{j+1} - (y-1)^k \right).
$$
Written in this way, the coefficients of $\rho_k(x)$ are closely related to
OEIS A108267 \cite{OEIS} (the latter having no ``sum free'' expression as well)
and shows that $\rho_k(x)$ has exactly $k$ nonzero terms.
It is also not too difficult to show that all terms of $\sigma_k(x)$ but one, have the same sign.
So, by using Descartes' rule of signs, we already have that $p_k(x)$ has at most $2k+2$ positive roots.
However, we require a stronger statement.

\begin{lemma}\label{l:5}
	The leading coefficient of $p_k(x)$ is $-1$. If $k$ is odd then $p_k(0) > 0$,
	otherwise $p_k(0) < 0$. In particular, $p_k(x)$ has at least one negative root.
\end{lemma}
\begin{proof}
	The leading coefficient of $p_k(x)$ is $-1$ if and only if $\sigma_k(x)$ is monic.
	Considering the terms of the sum defining $\sigma_k(x)$, the coefficient of $x^{k^2-1}$ in
	the expression
	$$
	x^{j+1}(x-1)^{k-j-1}(1+x+\cdots+x^{k-1})^k
	$$
	is $j+1$, so the leading coefficient of $\sigma_k(x)$ is
	$$
	\sum_{j=0}^{k-1} (-1)^j \binom{k+1}{j+2}(j+1) = 1\,.
	$$
	For the second part of the claim, note that $p_k(0)=\alpha_k\rho_k(0)-\sigma_k(0)$.
	As for $\rho_k(0)$ we have that
	\begin{align*}
	\rho_k(0) & = \sum_{j=0}^{k-1} (-1)^{j+k+1}j! C(k,k+1,j+2)\\
	          & = \sum_{j=0}^{k-1} (-1)^{j+k+1}j! \left((kj+k)C(k,k,j+2)+kC(k,k,j+1)\right)\\
	          & = (-1)^{k+1}kC(k,k,1)  +  (k-1)!k^2C(k,k,k+1)\\
	          & = (-1)^{k+1}kC(k,k,1) + 0\\
	          & = (-1)^{k+1}k^2
	\end{align*}
	while $\sigma_k(0) = (-1)^{k+1} k$.
	Thus, we must show that $\alpha_k > 1/k$.
	Indeed, this holds from \eqref{e:basic} and since $\varphi_k < k/(k+1)$.
\end{proof}

\begin{conjecture}\label{conj:real-roots}
	$p_k(x)$ has at most two real roots in $(0,1)$, counting multiplicity.
\end{conjecture}

\begin{lemma}\label{l:6}
	Conjecture \ref{conj:real-roots} implies Conjecture \ref{conj:fk}.
\end{lemma}
\begin{proof}
	We use a similar argument as in \cite{boppana-2023}.
	Assume that $p_k(x)$ has at most two real roots in $(0,1)$, counting multiplicity.
	By Rolle’s theorem, applied $k+1$ times, it follows that $f_k(x)$ has at
	most $k+3$ roots in $[0, 1]$, counting multiplicity. By Corollary \ref{coro:zero-root},
	there is a root of multiplicity $k$ at $0$. By Lemma \ref{l:1}, there is a root at $1$ and a double root at $\varphi_k$. Thus we have found all $k+3$ roots of $f_k(x)$ in $[0, 1]$.
	Because $f_k(x)$ has a double root at $\varphi_k$, it is either all nonnegative or all
	non-positive on $[0, 1]$. By Lemma \ref{l:2}, it must be all nonnegative on $[0,1]$.
\end{proof}
Observe that the proof of Lemma \ref{l:6} shows that Conjecture \ref{conj:real-roots} is equivalent to the same
conjecture with {\em at most} replaced with {\em exactly}.
Table \ref{table:pk} list $p_k(x)$ explicitly for $2 \le k \le 6$
where we have written $\alpha=\alpha_k$ for clarity.
A Python script generating $p_k(x)$ for a given $k$ can be obtained from
\url{https://github.com/raphaelyuster/almost-k-union-closed/blob/main/polynomial.py}.

Boppana observed that $p_2(x)$ has exactly two distinct real roots in $(0,1)$, both simple.
This can also be observed from Table \ref{table:pk} using Descartes' rule of signs.
We show that $p_3(x)$ and $p_4(x)$ have at most two real roots in $(0,1)$, counting multiplicity.

\begin{table}
	\renewcommand{\arraystretch}{1.2}
	\begin{tabular}{c||p{15cm}}
		$k$ &  $p_k(x)$ \\
		\hline
		$2$ & $(-4\alpha+2)+3x-4\alpha x^2-x^3$\\
		\hline 
		$3$ & $(9\alpha-3)-6x-10x^2+(63\alpha-6)x^3-3x^4+2x^5+9\alpha x^6-x^8$  \\
		\hline
		$4$ &  $(-16\alpha+4)+10x+20x^2+35x^3+ (-496\alpha+40)x^4+44x^5+40x^6 +25x^7+(-496\alpha+20)x^8+10x^9+4x^{10}+5x^{11}-16\alpha x^{12}-x^{15}$\\
		\hline
		$5$ &  
		$(25\alpha-5)-15x-35x^2-70x^3-126x^4+(3025\alpha-185)x^5-255x^6-320x^7-365x^8-371x^9+(9525\alpha-365)x^{10}-320x^{11}-255x^{12}-185x^{13}-131x^{14}+(3025\alpha-70)x^{15}-35x^{16}-15x^{17}-5x^{18}+4x^{19}+
		25\alpha x^{20}-x^{24}$\\
		\hline
		$6$ & 
		$(-36\alpha+6)+21x+56x^2+126x^3+252x^4+462x^5+(-16416\alpha+756)x^6+1161x^7+1666x^8+2247x^9+2856x^{10}+
		3416x^{11}+(-123516\alpha+3906)x^{12}+4221x^{13}+4332x^{14}+4221x^{15}+3906x^{16}+3451x^{17}+(-123516\alpha+2856)x^{18}+2247x^{19}+1666x^{20}+1161x^{21}+756x^{22}+441x^{23}+(-16416\alpha+252)x^{24}+126x^{25}+56x^{26}+21x^{27}+6x^{28}+7x^{29}-36\alpha x^{30}-x^{35}$
	\end{tabular}
	\caption{$p_k(x)$ for $k=2,\ldots,6$. For notational clarity, $\alpha=\alpha_k$.}
	\label{table:pk} 
\end{table}

\begin{proposition}\label{prop:p3}
	$p_3(x)$ has at most two real roots in $(0,1)$, counting multiplicity.
\end{proposition}
\begin{proof}
	By Table \ref{table:pk}, and since $\alpha_3 \approx 0.4655$, we have that $p_3(1)=81\alpha_3 - 27 > 0$. Since its degree is even and its leading
	coefficient is negative, this implies that $p_3(x)$ has a root larger than $1$.
	By Lemma \ref{l:5}, $p_3(x)$ has a negative root. It therefore suffices to prove that $p_3(x)$ has
	at most four real roots counting multiplicity.
	To this end, it suffices to prove that the third derivative of $p_3(x)$ has precisely one simple real root.
	The third and fourth derivatives of $p_3(x)$ are:
	\begin{align*}
	   {p_3}^{(3)}(x) & =  (378\alpha_3-36)-72x+120x^2+1080\alpha_3 x^3-336x^5\,;\\
	   {p_3}^{(4)}(x) & =  -72+240x+3240\alpha x^2 - 1680x^4\,.
	\end{align*}
	We show that ${p_3}^{(4)}(x)$ has exactly four real roots, all simple:
	\begin{alignat*}{5}
	&{p_3}^{(4)}\left(-0.9\right) & ~= &~ \tfrac{9}{125}(36450\alpha_3 - 19309) &~ < 0\,,\\
	&{p_3}^{(4)}\left(-0.8\right) & ~= &~ \tfrac{216}{125}(1200\alpha_3-551) & > 0\,,\\
	&{p_3}^{(4)}\left(0\right)    & ~= &~ -72                & < 0\,,\\
	&{p_3}^{(4)}\left(0.5\right)  & ~= &~ 810\alpha_3-57     & > 0\,.
\end{alignat*}
Denoting the roots of ${p_3}^{(4)}(x)$ by $\gamma_1 < \gamma_2 < \gamma_3 < \gamma_4$, we have
$\gamma_1 \in (-0.9,-0.8)$,
$\gamma_2 \in (-0.8,0)$,
$\gamma_3 \in (0,0.5)$,
$\gamma_4 \in (0.5,\infty)$.
	
As the leading coefficient of ${p_3}^{(3)}(x)$ is negative, it must be that
$\gamma_1,\gamma_3$ are local minima of ${p_3}^{(3)}(x)$ and $\gamma_2,\gamma_4$ are local maxima of
${p_3}^{(3)}(x)$. To show that ${p_3}^{(3)}(x)$ only has one simple real root, it
suffices to prove that the value of ${p_3}^{(3)}(x)$ at both local minima is positive.
First observe that ${p_3}^{(3)}{(0)} = 378\alpha_3-36 > 100$.
Now, for every $x \in [0,0.5]$ we have that
$$
{p_3}^{(3)}(x) \ge {p_3}^{(3)}(0) - 72 \cdot \tfrac{1}{2} - 336 \cdot \tfrac{1}{32} > 50\;.
$$
As $\gamma_3 \in (0,0.5)$, we have that ${p_3}^{(3)}{(\gamma_3)} > 0$.
We next show that ${p_3}^{(3)}{(\gamma_1)} > 0$.
$$
{p_3}^{(3)}(-0.9) = \tfrac{9}{6250}(225281 - 284250\alpha_3) > 133\;.
$$
For every $x \in [-0.9,-0.8]$ we have
\begin{align*}
{p_3}^{(3)}(x) - {p_3}^{(3)}(-0.9)
& = -72\left(x+\tfrac{9}{10}\right)+120\left(x^2-\tfrac{81}{100}\right)+1080\alpha_3\left(x^3 + \tfrac{729}{1000}\right) - 336\left(x^5+\tfrac{59049}{100000}\right) \\
& \ge -72\left(-\tfrac{8}{10}+\tfrac{9}{10}\right)+120\left(\tfrac{16}{25}-\tfrac{81}{100}\right)
- 336\left(-\tfrac{1024}{3125}+\tfrac{59049}{100000}\right) = -\tfrac{724401}{6250}\\
&  > -116\;.
\end{align*}
As $\gamma_1 \in (-0.9,-0.8)$, we have ${p_3}^{(3)}{(\gamma_1)} > 0$.
\end{proof}

\begin{proposition}\label{prop:p4}
	$p_4(x)$ has at most two real roots in $(0,1)$, counting multiplicity.
\end{proposition}
\begin{proof}
	By Lemma \ref{l:5}, $p_4(x)$ has a negative root. It therefore suffices to prove that $p_4(x)$ has
	at most three real roots, counting multiplicity.
	
	There are two distinct ways to prove this fact. The one we will not pursue in detail here, is by considering the signs of the discriminants of all the derivatives of $p_4(x)$.
	It turns out that the sign pattern of these discriminants is $(+,+,-,-,-,+,-,-,-,+,-,-,-,0,+)$
	where the $i$'th coordinate (starting at $i=0$) is the sign of the discriminant of ${p_4}^{(i)}(x)$.
	Recalling that the discriminant of a (real, univariate) polynomial is zero if and only if it has a multiple root and otherwise it is positive if and only if the number of non-real roots (counting multiplicity) is a multiple of $4$, we easily obtain that the number of real roots of the derivatives follows the sequence $(3,2,3,2,1,2,3,2,1,2,3,2,1,2,1)$ where the $i$'th coordinate (starting at $i=0$) is the number of real roots of ${p_4}^{(i)}(x)$. This is seen, starting as follows: the $14$'th derivative is a linear polynomial so has precisely one real root. The $13$'th derivative has discriminant $0$, and has a multiple root (at $x=0$, in fact). The $12$'th derivative has negative discriminant, so it must have two conjugate non-real roots, and one real root. The $11$'th derivative has negative discriminant, so again has only
	two non-real conjugate roots, and hence two real roots. Continuing this way, we see that for
	this particular sign pattern of discriminants, the number of real roots of ${p_4}^{(i)}(x)$
	is uniquely determined from the number of real roots of ${p_4}^{(i+1)}(x)$, from the sign of
	the discriminant of ${p_4}^{(i)}(x)$, from the fundamental theorem of algebra, and from the fact that the number or real roots of a polynomial is at most one larger than the number of real roots of its derivative.
	Finally, we obtain that the number of real roots of ${p_4}^{(0)}(x)$, i.e. ${p_4}(x)$, is $3$.
	A Maple worksheet computing these discriminant signs is available at
	\url{https://github.com/raphaelyuster/almost-k-union-closed/blob/main/p4.mw}.
	Observe that each discriminant is an integer polynomial in $\alpha_4$, and hence an integer polynomial in $\varphi_4$.
	But recall that $\varphi_4^4=1-\varphi_4$, so each of these discriminants can be reduced to an {\em integer cubic} polynomial in $\varphi_4$ (the polynomial $x^4+x-1$ is irreducible over ${\mathbb Q}$). Thus, the discriminant signs are easy to obtain by simply assigning $\varphi_4$ into explicit integer cubic polynomials.
	
	A more direct approach is similar to the one in Proposition \ref{prop:p4} and requires considering
	a few derivatives (but not all). A detailed rigorous account is given in Appendix
	\ref{appendix:p4} where we prove that
	the real-root pattern of the derivatives of $p_4(x)$ is $(3,2,3,2,1,2,3,2,1,2,3,2,1,2,1)$ as stated above.
\end{proof}
By Lemma \ref{l:6}, Proposition \ref{prop:p3} and Proposition \ref{prop:p4}, we have
\begin{corollary}\label{coro:234}
	$f_k(x)$ is nonnegative in $[0,1]$ for $k=3,4$ (and for $k=2$, as shown in \cite{AHS-2022,boppana-2023}).
\end{corollary}

\section{The multidimensional Chase-Lovett function}\label{sec:multi}

For $x \in (0,1)$, let
$$
F_k(x) \coloneqq \frac{h(x^k)}{x^{k-1}h(x)}\;.
$$
Let $\varphi=\varphi_2 = \frac{\sqrt{5}-1}{2}$. The following lemma is proved in \cite{CL-2022}.
\begin{lemma}[\cite{CL-2022}]\label{l:cl}
	For $x,y \in [0,1]$ it holds that
	\[
	\pushQED{\qed} 
	h(xy) \ge \frac{1}{2\varphi}\left(xh(y) + yh(x)\right)\,.
	\qedhere \popQED
	\]
\end{lemma}
\noindent
Let
$$
\mu_k \coloneqq
\begin{cases}
	\frac{1}{\alpha_k} & {\rm if}~~ 2 \le k \le 4\,;\\
	\frac{2^p-q}{2^p \varphi^p}+ \frac{q}{2^p \varphi^{p+1}} & {\rm if}~~ k \ge 5,\, p=\lfloor \log_2(k)\rfloor,\, q=k-2^p\,.
\end{cases}
$$
We apply Lemma \ref{l:cl} and our results from the previous section to lower-bound $F_k(x)$. 
\begin{lemma}\label{l:lb}
	For $x \in (0,1)$ we have
	$$
	F_k(x) \ge \mu_k\;.
	$$
\end{lemma}
\begin{proof}
We proceed by induction on $k$, where $k=2,3,4$ hold by Corollary \ref{coro:234}.

For the sake of the induction, observe also that the expression defining $\mu_k$ for $k \ge 5$ can be naively used for $k=2,3,4$.
Indeed, for $k=2$ the expression equals $1/\varphi=1/\alpha_2$, for $k=3$ the expression is
$1/2\varphi +1/2\varphi^2 = 2.118.. < 2.148.. = 1/\alpha_3$
and for $k=4$ the expression is $1/\varphi^2 =2.618.. < 2.630.. = 1/\alpha_4$.
Assume that $k=2^p+q \ge 5$ where $0 \le q < 2^p$ and that the lemma holds for values smaller than $k$.
By Lemma \ref{l:cl} we have
\begin{align*}
   F_k(x) & = \frac{h(x^k)}{x^{k-1}h(x)}\\
   & = \frac{h(x^{\lfloor k/2 \rfloor}x^{\lceil k/2 \rceil})}{x^{k-1}h(x)}\\
   & \ge \frac{1}{2\varphi}\left(\frac{x^{\lfloor k/2 \rfloor}h(x^{\lceil k/2 \rceil})+x^{\lceil k/2 \rceil}h(x^{\lfloor k/2 \rfloor}) }{x^{k-1}h(x)}\right)\\
   & = \frac{1}{2\varphi} \left( F_{\lceil k/2 \rceil}(x)+F_{\lfloor k/2 \rfloor}(x)\right)\\
   & \ge \frac{1}{2\varphi} \left(\frac{2^{p-1}-\lceil q/2 \rceil}{2^{p-1} \varphi^{p-1}}+ \frac{\lceil q/2 \rceil}{2^{p-1} \varphi^{p}}+ \frac{2^{p-1}-\lfloor q/2 \rfloor}{2^{p-1} \varphi^{p-1}}+ \frac{\lfloor q/2 \rfloor}{2^{p-1} \varphi^{p}}\right)\\
   & = \frac{2^p-q}{2^p \varphi^p}+ \frac{q}{2^p \varphi^{p+1}}\;.
 \end{align*}
\end{proof}
\begin{lemma}\label{l:mu}
	Let $1 \le m \le k-2$. Then, $\mu_{k-m}/(k-m) > \mu_k/k$.
\end{lemma}
\begin{proof}
	By telescoping product and induction, it suffices to prove that for all $k \ge 2$,
	$\mu_{k-1}/\mu_k > (k-1)/k$.
	
	For $k=3$ we have $\mu_2/\mu_3 = \alpha_3/\alpha_2 = 0.4655../0.6180.. \approx 0.7523.. > 2/3$.
	For $k=4$ we have $\mu_3/\mu_4 = \alpha_4/\alpha_3 = 0.3802../0.4655.. \approx 0.8167.. > 3/4$.
	For $k=5$ we have $\mu_4/\mu_5 = 1/(\alpha_4(3/4\varphi^2 + 1/4\varphi^3)) = 1/(0.3802.. \cdot  3.0229.. \approx 0.8700..) > 4/5$. Se we may now assume that $k \ge 6$.
	
	Consider first the case that $k=2^p+q$ and $1 \le q < 2^p$, so $k-1=2^p+q-1$.
	We have
	\begin{align*}
	\frac{\mu_{k-1}}{\mu_k} &= \frac{\frac{2^p-q+1}{2^p \varphi^p}+ \frac{q-1}{2^p \varphi^{p+1}} }
	{\frac{2^p-q}{2^p \varphi^p}+ \frac{q}{2^p \varphi^{p+1}} }\\
	& = \frac{\varphi(2^p-q+1)+q-1}{\varphi(2^p-q)+q}\\
	& = 1 - \frac{1-\varphi}{\varphi(2^p-q)+q}
	\end{align*}
so it remains to prove that
$$
\frac{\varphi(2^p-q)+q}{1-\varphi} =  \frac{\varphi k -2q\varphi+q}{1-\varphi} > k
$$
which is equivalent to $k > q$, which indeed holds.

Consider next the case where $k=2^p$, so $k-1=2^{p-1}+q$ where $q=2^{p-1}-1$. We have
$$
\frac{\mu_{k-1}}{\mu_k} = \frac{\frac{1}{2^{p-1} \varphi^{p-1}}+ \frac{2^{p-1}-1}{2^{p-1} \varphi^{p}} }
	{\frac{1}{\varphi^p}}
	= \frac{\varphi +2^{p-1}-1}{2^{p-1}}
	= \frac{\varphi +k/2-1}{k/2} > 1 - \frac{1}{k}\;.
$$
\end{proof}
\noindent
Let $g(x)=h(x)/x$ and let $M_k : (0,1)^k \rightarrow {\mathbb R}_{\ge 0}$ be defined as
$$
M_k(x_1,\ldots,x_k) \coloneqq \frac{g(\prod_{i=1}^k x_i)}{\sum_{i=1}^k g(x_i)}\;.
$$
The function $M_2$ plays a crucial role in the proof of \cite{CL-2022}, and so does its generalization here.
Notice that $M_k$ is smooth in $(0,1)^k$. By routine calculations (e.g. l’Hospital’s rule) it
is easily shown:
\begin{lemma}\label{l:mk-closed}
	$M_k(x)$ is extended continuously to $[0,1]^k$ as follows:
	Suppose $(x_1,\ldots,x_k)$ contains $\ell$ zeroes and $m$ ones, where $\ell+m > 0$.
	If $\ell > 0$ or $m \ge k-1$, then $M_k(x_1,\ldots,x_k)=1$.
	Otherwise, suppose that $x_{i_1},\ldots,x_{i_{k-m}}$ are not $1$, then,
	$M_k(x_1,\ldots,x_k)=M_{k-m}(x_{i_1},\ldots,x_{i_{k-m}})$. \qed
\end{lemma}
\noindent
We call a point in $[0,1]^k$ {\em diagonal} if it is supported on $\{t,1\}$ for some $t \in (0,1)$.
\begin{lemma}\label{l:mk-bounded}
	$\mu_k/k \le M_k < 1$ in $(0,1)^k$. Furthermore, every minimum of $M_k$ in $[0,1]^k$ is obtained
	in some diagonal point.
\end{lemma}
\begin{proof}
	The proof proceeds by induction on $k$. The case $k=2$ is proved in \cite{CL-2022}
	and the unique minimum is at $(\varphi,\varphi)$ where $M_2(\varphi,\varphi)=1/2\varphi=1/2\alpha_2 = \mu_2/2$. Let $k \ge 3$ and assume the lemma holds for values smaller than $k$.
	In $(0,1)^k$ we have that
	\begin{align*}
	M_k(x_1,\ldots,x_k) & = \frac{g(\prod_{i=1}^k x_i)}{\sum_{i=1}^k g(x_i)}\\
	& = \frac{g(x_1x_2 \prod_{i=3}^k x_i)}{g(x_1)+g(x_2)+\sum_{i=3}^k g(x_i)}\\
	& < \frac{g(x_1x_2 \prod_{i=3}^k x_i)}{g(x_1x_2)+\sum_{i=3}^k g(x_i)}\\
	& = M_{k-1}(x_1x_2,x_3,\ldots,x_k)\\
	& < 1\;.
	\end{align*}

	By Lemma \ref{l:mk-closed}, the values at boundary points are either $1$, or of  the form
	$M_{k-m}(x_1,\ldots,x_{k-m})$ for some point $(x_1,\ldots,x_{k-m}) \in (0,1)^{k-m}$
	with $1 \le m \le k-2$. As we already proved that $M_k < 1$ in $(0,1)^k$, only the latter points are ``potential'' minimum points. Suppose first that $(x_1,\ldots,x_{k-m})$ is not a diagonal point. By the induction hypothesis, it is not a minimum point of $M_{k-m}$. So there exist some
	$\delta_1,\ldots,\delta_{k-m}$ (some may be negative) such that $x_i + \delta_i \in (0,1)$ for $i \in [k-m]$
	and such that $M_{k-m}(x_1,\ldots,x_{k-m}) > M_{k-m}(x_1+\delta_1,\ldots,x_{k-m}+\delta_{k-m})$.
	Since $M_k(x_1+\delta_1,\ldots,x_{k-m}+\delta_{k-m},1,\ldots,1) = M_{k-m}(x_1+\delta_1,\ldots,x_{k-m}+\delta_{k-m})$, we have that
	$M_k$ does not attain minimum at the stated boundary point.
	Consider next the case that $x_i=t$ for $i \in [k-m]$ and some $t \in (0,1)$.
	Then $M_{k-m}(t,\ldots,t) = F_{k-m}(t)/(k-m) \ge \mu_{k-m}/(k-m)$ by Lemma \ref{l:lb}
	and $\mu_{k-m}/(k-m) \ge \mu_k/k$ by Lemma \ref{l:mu}.
	
	It remains to consider the case where the minimum is attained at an internal point.
	Here we use the same approach as in \cite{CL-2022}.
	Assume that $M_k$ is minimized at some point $(x_1^*, \ldots,x_k^*) \in (0,1)^k$, and let
	$\beta  = M_k(x_1^*, \ldots,x_k^*)$. Let
	$$
	G(x_1,\ldots,x_k) = g\left(\prod_{i=1}^k x_i\right) - \beta\left(\sum_{i=1}^k g(x_i)\right)\;.
	$$
	Then $G$ is nonnegative in $(0,1)^k$ and $(x_1^*, \ldots,x_k^*)=0$.
	Thus the partial derivatives of $G$ are zero at the minimum point:
	$$
	\frac{\partial G}{\partial x_i}(x_1^*, \ldots,x_k^*)=0 \quad {\rm for~all~} 1 \le i \le k\;.
	$$
	Evaluating the derivatives gives
	$$
	\frac{\partial G}{\partial x_i}(x_1, \ldots,x_k)=g'\left(\prod_{j=1}^k x_j\right)\frac{\prod_{j=1}^k x_j}{x_i}-\beta g'(x_i) \quad {\rm for~all~} 1 \le i \le k\;.
	$$
	Defining $Q(x)=x g'(x)$ we obtain that $Q(x_1^*)=Q(x_2^*)=\cdots = Q(x_k^*)$.
	Since $Q(x)=\log(1-x)/x$ is strictly decreasing, we must have $x_1^*=x_2^*=\cdots=x_k^*=t$ for some $t \in (0,1)$. But notice that in this case we have $M_k(t,\ldots,t) = F_k(t)/k \ge \mu_k/k$ by Lemma \ref{l:lb}.
\end{proof}

\begin{corollary}\label{coro:main}
	For $(x_1,\ldots,x_k) \in [0,1]^k$ it holds that
	$$
	h\left(\prod_{i=1}^k x_i\right) \ge \frac{\mu_k}{k}\left(\sum_{i=1}^k h(x_i)\cdot\prod_{j \in [k] \setminus i}
	x_j \right).
	$$
\end{corollary}
\begin{proof}
	For $k=2$ this is just Lemma \ref{l:cl}. Assume that $k \ge 3$ and that the claim holds for smaller $k$.
	If $(x_1,\ldots,x_k)$ is an internal point, then the claim follows from Lemma \ref{l:mk-bounded}.
	If $(x_1,\ldots,x_k)$ contains a zero, then the claim amount to $0=0$. Otherwise, we may assume that $x_k=1$.
	In this case we have by induction that
	$$
	h\left(\prod_{i=1}^k x_i\right) = h\left(\prod_{i=1}^{k-1} x_i\right) \ge
	 \frac{\mu_{k-1}}{k-1}\left(\sum_{i=1}^{k-1} h(x_i)\cdot\prod_{j \in [k-1] \setminus i}
	x_j \right)= \frac{\mu_{k-1}}{k-1}\left(\sum_{i=1}^{k} h(x_i)\cdot\prod_{j \in [k] \setminus i}
	x_j \right)
	$$
	and the claim follows from Lemma \ref{l:mu}.
\end{proof}

\section{Proofs of the main results}\label{sec:proofs}
For random variables $A_1,\ldots,A_k$ taking values in $\{0, 1\}^n$, let $A_{j,i} \in \{0,1\}$ be
the restriction of $A_j$ to the $i$'th coordinate and let $A_{j,<i} \in \{0, 1\}^{i-1}$ be the restriction
of $A_j$ to the first $i-1$ coordinates. Let $\cup_{j=1}^k A_j$ be the random variable taking values in $\{0, 1\}^n$ whose $i$'th coordinate is zero if and only if $A_{j,i}=0$ for all $j \in [k]$.
We similarly define $\cup_{j=1}^k A_{j,i} \in \{0,1\}$ and $\cup_{j=1}^k A_{j,<i} \in \{0,1\}^{i-1}$.
Given Corollary \ref{coro:main}, we can generalize Claim 4.1 of \cite{CL-2022}.
\begin{lemma}\label{l:main}
	Let $A_1,\ldots,A_k$ be mutually independent random variables taking values in
	$\{0, 1\}^n$. Assume for all $i \in [n]$ and $j \in [k]$ that $\Pr[A_{j,i} = 0] \ge p$.
	Then,
	$$
	H(\cup_{j=1}^k A_j) \ge \frac{p^{k-1} \mu_k}{k}\left(\sum_{j=1}^k H(A_j)\right).
	$$
\end{lemma}
\begin{proof}
	By the chain rule for entropy,
	$$
	H(\cup_{j=1}^k A_j) = \sum_{i=1}^n H(\cup_{j=1}^k A_{j,i}\,|\,\cup_{j=1}^k A_{j,< i})\;.
	$$
	By the data processing inequality,
	$$
	\sum_{i=1}^n H(\cup_{j=1}^k A_{j,i}\,|\,\cup_{j=1}^k A_{j,< i}) \ge \sum_{i=1}^n H(\cup_{j=1}^k A_{j,i}
	\,|\, A_{1,<i},A_{2,<i},\ldots,A_{k,<i})\;.
	$$
	Let $q_{j,i}(x) = \Pr[A_{j,i} = 0\,|\, A_{j,< i} = x]$ (here $x \in \{0,1\}^{i-1}$). By Corollary \ref{coro:main},
	\begin{align*}
	& H(\cup_{j=1}^k A_{j,i}
	\,|\, A_{1,<i}=x_1,A_{2,<i}=x_2,\ldots,A_{k,<i}=x_k)\\
	=
	& h\left(\prod_{j=1}^k q_{j,i}(x_j)\right) 
	\ge \frac{\mu_k}{k}\left(\sum_{j=1}^k h(q_{j,i}(x_j)) \cdot \prod_{\ell \in [k] \setminus j}
	q_{\ell,i}(x_\ell) \right).
	\end{align*}
	Averaging over $A_{1,<i},\ldots,A_{k,<i}$ which are mutually independent gives
	\begin{align*}
	H(\cup_{j=1}^k A_{j,i}
	\,|\, A_{1,<i},A_{2,<i},\ldots,A_{k,<i})
	& \ge \frac{\mu_k}{k}\left(\sum_{j=1}^k
	{\mathbb E}_{A_{j,<i}}[h(q_{j,i}(A_{j,<i}))] \cdot
	\prod_{\ell \in [k] \setminus j}
	{\mathbb E}_{A_{\ell,<i}}[q_{\ell,i}(A_{\ell,<i})]
	\right)\\
	& = \frac{\mu_k}{k}\left(\sum_{j=1}^k
	H(A_{j,i}\,|\,A_{j,<i})\cdot \prod_{\ell \in [k] \setminus j}
	\Pr[A_{\ell,i}=0]
	\right).
	\end{align*}
	Since $\Pr[A_{j,i} = 0] \ge p$ we have
	$$
	\Pr\left[\cup_{j=1}^k A_{j,i}\right] \ge \frac{p^{k-1}\mu_k}{k}\left(\sum_{j=1}^k H(A_{j,i}\,|\,A_{j,<i})\right)\;.
	$$
	The lemma then follows by summing over $i \in [n]$.
\end{proof}

Prior to proving our main results, we define the constant $z_k$ stated in Theorem \ref{t:clg-345} and establish its correspondence with $\psi_k$. Let
$$
z_k \coloneqq 1 - {\mu_k}^{1/(1-k)}\;.
$$
\begin{proposition}\label{prop:zk}
	$z_k = \psi_k$ for $k=2,3,4$. Furthermore,
	$$
	z_k > \frac{\log k}{3k}\;, \qquad \frac{1}{2} < \frac{z_k}{\psi_k} \le 1\;, \qquad \lim_{k \rightarrow \infty} \frac{z_k}{\psi_k} = \frac{\log \frac{1}{\varphi}}{\log 2} \approx 0.6943\;.
	$$
\end{proposition}
\begin{proof}
	By the definitions of $z_k$, $\mu_k$, $\varphi_k$, $\alpha_k$, $\psi_k$ we have $z_2 = 1-\alpha_2=\psi_2$,
	$z_3=1-(\alpha_3)^{1/2} = 1-\varphi_3=\psi_3$ and $z_4 = 1-(\alpha_4)^{1/3} = 1-\varphi_4 = \psi_4$.
	
	By the definitions of $\varphi_k$ and $\alpha_k$, we have that $F_k(\varphi_k)=1/\alpha_k$.
	By Lemma \ref{l:lb},
	$$
	z_k = 1 - {\mu_k}^{1/(1-k)} \le 1-F_k(\varphi_k)^{1/(1-k)} = 1 - \left(\frac{1}{\alpha_k}\right)^{1/(1-k)} = 1- \alpha_k^{1/(k-1)} = 1-\varphi_k=\psi_k\;.
	$$
	
	Consider the function $(1-x)^k-x$ for which $\psi_k$ is a root in $(0,1)$. As this function is monotone decreasing in $(0,1)$, $\psi_k$ is its only root there.
	Since $(1-\log k/k)^k - \log k /k < 0$ for all $k \ge 3$, we have that
	$\psi_k < \log k/k$ for all $k \ge 3$. Notice also that for every $\eps \in (0,1)$,
	$(1-(1-\eps)\log k/k)^k - (1-\eps)\log k/k > 0$ for all sufficiently large $k$, thus
	$\psi_k =(1-o(1))\log k/k$. In fact, it is easily verified that $\epsilon=\frac{1}{3}$ works for all $k \ge 2$, hence $\psi_k \ge (2\log k)/(3k)$.
	
	Now suppose that $k=2^p+q$ where $p=\lfloor \log_2 k \rfloor$.
	Notice that since $\mu_k$ is increasing with $k$, we have that $z_k= 1 - {\mu_k}^{1/(1-k)}
	\ge 1 - {\mu_{2^p}}^{1/(1-k)} = 1-\varphi^{\lfloor \log_2 k \rfloor/(k-1)}$.
	
	Using the inequality $e^{-x} \le 1-x+x^2/2$ valid for all $x \ge 0$ we have
	\begin{align}
		\frac{z_k}{\psi_k} & \ge \frac{1- \varphi^{\lfloor \log_2 k \rfloor/(k-1)}}{\log k/k} \label{e:first}\\	
		& = \frac{1- e^{-\log(1/\varphi)\frac{\lfloor \log_2 k \rfloor}{k-1}}}{\log k/k} \nonumber\\
		& \ge \frac{1- e^{-\frac{\log(1/\varphi)}{\log 2}\frac{(\log k)-1}{k-1}}}{\log k/k} \nonumber\\
		& \ge \frac{\frac{\log(1/\varphi)}{\log 2}\frac{(\log k)-1}{k-1}-
		\frac{\log^2(1/\varphi)}{2\log^2 2}\frac{((\log k)-1)^2}{(k-1)^2}}{\log k/k} \nonumber\\
	    & \ge \frac{\frac{\log(1/\varphi)}{\log 2}((\log k) -1)-\frac{\log^2(1/\varphi)}{2\log^2 2}
	    	\frac{((\log k)-1)^2}{(k-1)}}{\log k}\;. \label{e:last}
	\end{align}
We immediately obtain from the last inequality that
$$
\liminf_{k \rightarrow \infty} \frac{z_k}{\psi_k} \ge \frac{\log \frac{1}{\varphi}}{\log 2}\;.
$$
To see that this is, in fact, a limit, just repeat the last series of inequalities by (i) reversing each inequality; (ii) using the lower bound $\psi_k \ge (1-o(1)\log k/k$; (iii) using the upper bound
$z_k= 1 - {\mu_k}^{1/(1-k)} \le 1 - {\mu_{2^{p+1}}}^{1/(1-k)} = 1-\varphi^{1+ \lfloor \log_2 k \rfloor/(k-1)}$;
(iv)  apply the inequality $e^{-x} \ge 1-x$.

Finally, it is easily verified that \eqref{e:first} is larger than $\frac{1}{2}$ for $k \le 100$ and
\eqref{e:last} is larger than $\frac{1}{2}$ for $k > 100$. Thus, $z_k/\psi_k > \frac{1}{2}$
and $z_k > \psi_k/2 \ge (\log k)/(3k)$.
\end{proof}

\begin{proof}[Proof of Theorems \ref{t:clg-345} and \ref{t:clg}]
	Let $\F \subseteq 2^{[n]}$, $\F \neq \{\emptyset\}$ be a $(1-\eps)$-approximate $k$-union closed set system, where $0 \le \eps <\frac{1}{2}$.
	Let $p_i$ be the fraction of sets in $\F$ that do not contain $i$ and let $p = \min_{i \in [n]} p_i$.
	Let $A_1,\ldots,A_k$ be a $k$-tuple of sets of $\F$, where $A_j$ is chosen uniformly and independently of the other sets. By Lemma \ref{l:main} we obtain:
	$$
	H(\cup_{j=1}^k A_j) \ge \frac{p^{k-1} \mu_k}{k}\left(\sum_{j=1}^k H(A_j)\right) = p^{k-1}\mu_k \log |\F|\;.
	$$
	As in \cite{CL-2022}, we show that $H(\cup_{j=1}^k A_j)$ cannot be much larger than $\log |\F|$.
	Let $I$ be the indicator for the event $\cup_{j=1}^k A_j \in \F$ where by assumption
	$\Pr[I = 1] \ge 1-\epsilon$. We have
	$$
	H(\cup_{j=1}^k A_j) \le H(\cup_{j=1}^k A_j, I) = H(I) + H(\cup_{j=1}^k A_j \,|\, I = 0)
	\Pr[I = 0] + H(\cup_{j=1}^k A_j \,|\, I=1)Pr[I = 1]\;.
	$$
	We bound the terms in the last inequality. Since $I \in \{0,1\}$, and $\Pr[I = 0] \le \eps  < \frac{1}{2}$, we have $H(I) \le h(\eps) \le  2\eps \log(1/\eps)$.
	Also note that $H(\cup_{j=1}^k A_j\,|\, I = 0) \le H(A_1,A_2,\ldots,A_k \,|\,I = 0)\le k \log |F|$.
	Finally, notice that $(\cup_{j=1}^k A_j \,|\, I= 1)$ is a distribution supported on $\F$ and so
	$H(\cup_{j=1}^k A_j \,|\, I=1) \le \log |\F|$. We therefore have
	$$
	 p^{k-1}\mu_k \log |\F| \le H(\cup_{j=1}^k A_j) \le 2\eps \log(1/\eps)+(1+k\epsilon)\log |\F|
	$$
	from which we immediately obtain
	$$
	1-p \ge 1 - {\mu_k}^{1/(1-k)} - \left(k\eps + \frac{2\eps \log(1/\eps)}{\log |\F|}\right)^{1/(k-1)}
	= z_k - \left(k\eps + \frac{2\eps \log(1/\eps)}{\log |\F|}\right)^{1/(k-1)}\;.
	$$
	Theorems \ref{t:clg-345} and \ref{t:clg} now follow from Proposition \ref{prop:zk}.
\end{proof}
Finally, by Lemma \ref{l:6}, Conjecture \ref{conj:real-roots} implies Conjecture \ref{conj:fk},
and Conjecture \ref{conj:fk} implies the validity of Corollary \ref{coro:234} for all $k$ (not just $k=2,3,4$),
which in turn, means that we can define $\mu_k=1/\alpha_k$ for all $k$ (not just $k=2,3,4$), which implies
Conjecture \ref{conj:clg}. Stated directly: if $p_k(x)$ has at most two real roots in $(0,1)$, then
Conjecture \ref{conj:clg} holds.

\appendix

\section{The real root pattern of the derivatives of $p_4$}\label{appendix:p4}

we prove that the number of real roots of the derivatives of $p_4$ follows the sequence
$$
(3,2,3,2,1,2,3,2,1,2,3,2,1,2,1)
$$
where the $i$'th coordinate (starting at $i=0$) is the number of real roots of ${p_4}^{(i)}(x)$.  For referential convenience, the derivatives of interest are:
\begin{align*}
	{p_4}^{(1)} & =
	{\scriptstyle 10+40x+105x^2+4(-496\alpha_4+40)x^3+220x^4+240x^5+175x^6+8(-496\alpha_4+20)x^7+90x^8+40x^9+55x^{10}-192\alpha_4x^{11}-15x^{14}};\\
	{p_4}^{(2)} & =
	{\scriptstyle
	 40+210x+12(-496\alpha_4+40)x^2+880x^3+1200x^4+1050x^5+56(-496\alpha_4+20)x^6+720x^7+360x^8+550x^9-2112
	\alpha_4x^{10}-210x^{13}};\\
	{p_4}^{(4)} & =
	{\scriptstyle
	-11904\alpha_4+960+5280x+14400x^2+21000x^3+1680(-496\alpha_4+20)x^4+30240x^5+20160x^6+39600x^7-
	190080\alpha_4x^8-32760x^{11}};\\
	{p_4}^{(5)} & =
	{\scriptstyle
    5280+28800x+63000x^2+6720(-496\alpha_4+20)x^3+151200x^4+120960x^5+277200x^6-1520640\alpha_4x^7
	-360360x^{10}};\\
	{p_4}^{(6)} & =
	{\scriptstyle
	28800+126000x+20160(-496\alpha_4+20)x^2+604800x^3+604800x^4+1663200x^5-10644480\alpha_4 x^6-3603600 x^9};\\
	{p_4}^{(8)} & =
	{\scriptstyle
	-19998720\alpha_4+806400+3628800x+7257600x^2+33264000x^3-319334400\alpha_4x^4-259459200x^7};\\
	{p_4}^{(9)} & =
	{\scriptstyle
	3628800+14515200x+99792000x^2-1277337600 \alpha_4 x^3-1816214400x^6};\\
	{p_4}^{(10)} & =
	{\scriptstyle
	14515200 + 199584000x - 3832012800\alpha_4 x^2-10897286400 x^5} ;\\
	{p_4}^{(12)} & =
	{\scriptstyle
	-7664025600\alpha_4-217945728000x^3} ;
\end{align*}

Clearly ${p_4}^{(13)}(x)$ is a parabola with a double root at $x=0$ and
${p_4}^{(14)}(x)$ is linear, so has a single root.
Observing the cubic ${p_4}^{(12)}(x)$, we see that it has one real root.
This implies that ${p_4}^{(10)}(x)$ has at most three real roots. Indeed, it has three since
$\alpha_4 \approx 0.3802$ and
\begin{alignat*}{5}
	&{p_4}^{(10)}\left(-0.2\right) & ~= &~ -\tfrac{2739308544}{125}-153280512\alpha_4 & ~< 0\,,\\
	&{p_4}^{(10)}\left(0\right)    & ~= &~ 14515200               & > 0\,.
\end{alignat*}
Let $\gamma_{10,1} \in (-\infty,-0.2)$, $\gamma_{10,2} \in (-0.2,0)$, $\gamma_{10,3} \in (0,\infty)$ be the real roots of
${p_4}^{(10)}(x)$.

As ${p_4}^{(9)}(x)$ has even degree and negative leading coefficient, it must be that
$\gamma_{10,2}$ is a local minimum of ${p_4}^{(9)}(x)$.
To prove that ${p_4}^{(9)}(x)$ has at most two real roots, we show that
${p_4}^{(9)}(\gamma_{10,2}) > 0$.
Indeed, ${p_4}^{(9)}(0) = 3628800$.
Now, for every $x \in [-0.2,0]$ we have
\begin{align*}
	{p_4}^{(9)}(x) - {p_4}^{(9)}(0) & = 14515200x+99792000x^2-1277337600 \alpha_4 x^3-1816214400x^6\\
	& \ge  14515200\left(-\tfrac{1}{5}\right)-1816214400\left(\tfrac{1}{5^6}\right)\\
	&  > -3628800\;.
\end{align*}
As $\gamma_{10,2} \in (-0.2,0)$, we have that ${p_4}^{(9)}(\gamma_{10,2}) > 0$.
We have shown that ${p_4}^{(9)}(x)$ has at most two real roots. Indeed, it has two since
\begin{alignat*}{5}
	&{p_4}^{(9)}\left(0\right)    & ~= &~ 3628800               & > 0\,,\\
	&{p_4}^{(9)}\left(0.4\right)    & ~= &~ \tfrac{11226491136}{625}-\tfrac{408748032}{5}\alpha_4 & ~< 0\,.
\end{alignat*}
Let $\gamma_{9,1} \in (-\infty,0)$, $\gamma_{9,2} \in (0,0.4)$ be the real roots of ${p_4}^{(9)}(x)$.

As ${p_4}^{(8)}(x)$ has odd degree and negative leading coefficient, it must be that
$\gamma_{9,2}$ is a local maximum of ${p_4}^{(8)}(x)$.
To prove that ${p_4}^{(8)}(x)$ has at most one real root, we show that
${p_4}^{(8)}(\gamma_{9,2}) < 0$.
Indeed, ${p_4}^{(8)}(0) = -19998720\alpha_4+806400  < -6000000$.
Now, for every $x \in [0,0.4]$ we have
\begin{align*}
	{p_4}^{(8)}(x) - {p_4}^{(8)}(0) & = 3628800x+7257600x^2+33264000x^3-319334400\alpha_4x^4-259459200x^7\\
	& \le  3628800\left(\tfrac{2}{5}\right)+7257600\left(\tfrac{4}{25}\right)+33264000\left(\tfrac{8}{125}\right)\\
	&  < 5000000\;.
\end{align*}
As $\gamma_{9,2} \in (0,0.4)$, we have that ${p_4}^{(8)}(\gamma_{9,2}) < 0$.
We have shown that ${p_4}^{(8)}(x)$ has at most one real root.

As ${p_4}^{(8)}(x)$ has at most one real root, it follows that ${p_4}^{(6)}(x)$
has at most three real roots. Indeed, it has three since
\begin{alignat*}{5}
	&{p_4}^{(6)}\left(-0.15\right) & ~= &~ \tfrac{21901848684147}{1280000000}-\tfrac{2813835591}{12500}\alpha_4 &~ < 0\,,\\
	&{p_4}^{(6)}\left(0\right)    & ~= &~ 28800               & > 0\,.
\end{alignat*}
Let $\gamma_{6,1} \in (-\infty,-0.2)$, $\gamma_{6,2} \in (-0.15,0)$, $\gamma_{6,3} \in (0,\infty)$ be the real roots of ${p_4}^{(6)}(x)$.

As ${p_4}^{(5)}(x)$ has even degree and negative leading coefficient, it must be that
$\gamma_{6,2}$ is a local minimum of ${p_4}^{(5)}(x)$.
To prove that ${p_4}^{(5)}(x)$ has at most two real roots, we show that
${p_4}^{(5)}(\gamma_{6,2}) > 0$.
Indeed, ${p_4}^{(5)}(0) = 5280$.
Now, for every $x \in [-0.15,0]$ we have
\begin{align*}
	{p_4}^{(5)}(x) - {p_4}^{(5)}(0) & =
	28800x+63000x^2+6720(-496\alpha_4+20)x^3+151200x^4+120960x^5\\
	& \quad +277200x^6-1520640\alpha_4x^7-360360x^{10}\\
	& \ge
	28800\left(-\tfrac{3}{20}\right)+120960\left(-\tfrac{3}{20}\right)^5
	-360360\left(-\tfrac{3}{20}\right)^{10}\\
	&  > -4400\;.
\end{align*}
As $\gamma_{6,2} \in (-0.15,0)$, we have that ${p_4}^{(5)}(\gamma_{6,2}) > 0$.
We have shown that ${p_4}^{(5)}(x)$ has at most two real roots.
Indeed, it has two since 
\begin{alignat*}{5}
	&{p_4}^{(5)}\left(0\right)    & ~= &~ 5280               & > 0\,,\\
	&{p_4}^{(5)}\left(0.25\right)    & ~= &~ \tfrac{2528848395}{131072}-\tfrac{834765}{16}\alpha_4 & ~< 0\,.
\end{alignat*}
Let $\gamma_{5,1} \in (-\infty,0)$, $\gamma_{5,2} \in (0,0.25)$ be the real roots of ${p_4}^{(5)}(x)$.

As ${p_4}^{(4)}(x)$ has odd degree and negative leading coefficient, it must be that
$\gamma_{5,2}$ is its local maximum.
To prove that ${p_4}^{(4)}(x)$ has at most one real root, we show that
${p_4}^{(4)}(\gamma_{5,2}) < 0$.
Indeed, ${p_4}^{(4)}(0) = -11904\alpha_4+960 < -3565$.
Now, for every $x \in [0,0.25]$ we have
\begin{align*}
	{p_4}^{(4)}(x) - {p_4}^{(4)}(0) & = 5280x+14400x^2+21000x^3+1680(-496\alpha_4+20)x^4+30240x^5+20160x^6\\
	&\quad +39600x^7-190080\alpha_4x^8-32760x^{11}\\
	& \le
	5280x+14400x^2+21000x^3+30240x^5+20160x^6+39600x^7\\
	&  < 5280\left(\tfrac{1}{4}\right)+14400\left(\tfrac{1}{4}\right)^2+21000\left(\tfrac{1}{4}\right)^3+30240\left(\tfrac{1}{4}\right)^5+20160\left(\tfrac{1}{4}\right)^6+39600\left(\tfrac{1}{4}\right)^7\\
	& < 2600\;.
\end{align*}
As $\gamma_{5,2} \in (0,0.25)$, we have that ${p_4}^{(4)}(\gamma_{5,2}) < 0$.
Hence, ${p_4}^{(4)}(x)$ has at most one real root.

As ${p_4}^{(4)}(x)$ has at most one real root, it follows that ${p_4}^{(2)}(x)$
has at most three real roots. Indeed, it has three since
\begin{alignat*}{5}
	&{p_4}^{(2)}\left(-0.2\right) & ~= &~ \tfrac{2882593792}{244140625}-\tfrac{2342362112}{9765625}\alpha_4 &~ < 0\,,\\
	&{p_4}^{(2)}\left(0\right)    & ~= &~ 40               & > 0\,.
\end{alignat*}
Let $\gamma_{2,1} \in (-\infty,-0.2)$, $\gamma_{2,2} \in (-0.2,0)$, $\gamma_{2,3} \in (0,\infty)$ be the real roots of ${p_4}^{(2)}(x)$.

As ${p_4}^{(1)}(x)$ has even degree and negative leading coefficient, it must be that
$\gamma_{2,2}$ is a local minimum of ${p_4}^{(1)}(x)$.
To prove that ${p_4}^{(1)}(x)$ has at most two real roots, we show that
${p_4}^{(1)}(\gamma_{2,2}) > 0$.
Indeed, ${p_4}^{(1)}(0) = 10$.
Now, for every $x \in [-0.2,0]$ we have
\begin{align*}
	{p_4}^{(1)}(x) - {p_4}^{(1)}(0) & =
	40x+105x^2+4(-496\alpha_4+40)x^3+220x^4+240x^5+175x^6\\
	& \quad +8(-496\alpha_4+20)x^7+90x^8+40x^9+55x^{10}-192\alpha_4x^{11}-15x^{14}\\
	& \ge
	40x+240x^5+40x^9-15x^{14}\\
	&  \ge
	40\left(-\tfrac{1}{5}\right)+240\left(-\tfrac{1}{5}\right)^5+40\left(-\tfrac{1}{5}\right)^9-15\left(-\tfrac{1}{5}\right)^{14}\\
	& > -9\;.
\end{align*}
As $\gamma_{2,2} \in (-0.2,0)$, we have that ${p_4}^{(1)}(\gamma_{2,2}) > 0$.
We have shown that ${p_4}^{(1)}(x)$ has at most two real roots.
Hence $p_4(x)$ has at most three real roots. By the comment after Lemma \ref{l:6}, it must have precisely three. 
\qed
\end{document}